\documentclass[a4paper,10pt]{article}
\usepackage[utf8]{inputenc}
\usepackage{amsmath}
\usepackage{amsthm}
\usepackage{amssymb}
\usepackage{graphicx}
\usepackage{mathtools}
\usepackage{enumerate}
\usepackage{color}
\usepackage{epsfig}
\usepackage{cancel}
\usepackage{epstopdf}
\usepackage{hyperref}
\usepackage{float}
\usepackage{enumitem} 
\usepackage{authblk}

\newenvironment{enum_ass}{%
	\begin{enumerate}[label=(\textit{\alph*}),ref=(\textit{\alph*})]}
	{\end{enumerate}}

\newcommand{\trasp}{^*}
\newcommand{\R}{\mathbb{R}}
\newcommand{\C}{\mathbb{C}}
\newcommand{\M}{\mathcal{M}}
\newcommand{\TM}{\mathcal{T}_Y\mathcal{M}}
\newcommand{\ee}{\mathrm{e}}
\newcommand{\PhiA}{\Phi_{\tau}^A}
\newcommand{\PhiG}{\Phi_{\tau}^G}
\newcommand{\phiG}{\widetilde{\Phi}_{\tau}^G}
\newcommand{\LL}{\mathcal{I}}
\newcommand{\dd}{\, \mathrm{d}}

\newtheorem{theorem}{Theorem}
\newtheorem{lemma}{Lemma}
\newtheorem{prop}{Proposition}
\newtheorem{hp}{Assumption}

\definecolor{dunkelgruen}{rgb}{0.2,0.6,0.4} 

\title{Convergence of a low-rank Lie--Trotter splitting \\ for stiff matrix differential equations}

\author[$\dagger$]{Alexander~Ostermann}
\author[$\dagger$]{Chiara~Piazzola}
\author[$\ddag$]{Hanna~Walach}
\affil[$\dagger$]{Institut f\"ur Mathematik, Universit\"at Innsbruck, Austria}
\affil[$\ddag$]{Mathematisches Institut, Universit\"at T\"ubingen, Germany}
\affil[ ]{\url{alexander.ostermann} , \url{chiara.piazzola@uibk.ac.at} \\ \url{walach@na.uni-tuebingen.de}}

\date{\today}

\begin{document}

\maketitle

\begin{abstract}
We propose a numerical integrator for determining low-rank approximations to solutions of large-scale matrix differential equations. The considered differential equations are semilinear and stiff. Our method consists of first splitting the differential equation into a stiff and a non-stiff part, respectively, and then following a dynamical low-rank approach. We conduct an error analysis of the proposed procedure, which is independent of the stiffness and robust with respect to possibly small singular values in the approximation matrix. Following the proposed method, we show how to obtain low-rank approximations for differential Lyapunov and for differential Riccati equations. Our theory is illustrated by numerical experiments.
\end{abstract}

\section*{Key words}
Matrix differential equation, differential Lyapunov equation, differential Riccati equation, dynamical low-rank approximation, low-rank splitting \\
MSC: 65L04, 65L20, 65M12, 65F30, 49J20

\section{Introduction}
Dynamical low-rank approximations of matrices are widely used for reducing models of large size. Such an approach has a broad variety of application areas, such as control theory, computer algebra, signal processing, machine learning, image compression, and quantum molecular systems, see, e.g., \cite{BennerMena2, S15, Nonnemacher08, lubich15tii} and references therein. We are interested here in particular in computing low-rank approximations to solutions of large-scale matrix differential equations. 

In this paper we consider a class of semilinear stiff matrix differential equations of the form 
\begin{align*}
\dot{X}(t) = AX(t)+X(t)A\trasp+G(t,X(t)), \qquad X(t_0) = X^0,
\end{align*}
where $X(t) \in \mathbb{C}^{m \times m}$, $G: [t_0,\infty) \times \mathbb{C}^{m \times m} \rightarrow \mathbb{C}^{m \times m}$ is nonlinear, and $A \in \mathbb{C}^{m \times m}$ is time invariant. In many applications, the matrix $A$ arises from the spatial discretization of a differential operator. Therefore, it gives rise to a stiff term. The nonlinearity $G$, however, is assumed to be non-stiff. The objective of this paper is to determine a low-rank approximation to the solution of the given matrix differential equation. 

A possible method for obtaining low-rank approximations to solutions of matrix differential equations is the dynamical low-rank approximation proposed in \cite{KL07}. This approach yields a differential equation for the approximation matrix on the low-rank manifold. Recently, an efficient integrator, the so-called projector-splitting integrator, was proposed in \cite{LO14} for computing the solution numerically. 
A comprehensive error analysis for this integration method is given in \cite{KLW16}. Note that the error bounds in this analysis depend on the Lipschitz constant of the right-hand side of the considered matrix differential equation, amongst others. Therefore, this proof does not extend to the present situation in an obvious way.  

In this work we propose a novel approach, which yields low-rank approximations for stiff matrix differential equations. Our method is derived in two steps. To handle the difficulty with the stiff part, we first split the matrix differential equation into its stiff part $AX+XA\trasp$ and the non-stiff nonlinearity $G$. Second, we follow the concept of the dynamical low-rank approximation for both arising subproblems. The linear subproblem can be solved exactly and efficiently by means of exponential integrators and the rank of this solution is preserved. The nonlinearity $G$ is integrated with the projector-splitting integrator \cite{LO14}. 
We conduct a convergence analysis for the proposed method where, because of the beneficial way of splitting, we succeed in giving error bounds which are independent of the norm of $A$.

Moreover, our integration method is independent of small singular values, which might appear in the approximation matrix. When following the original approach of \cite{KL07}, one would have to solve a modified differential equation whose right-hand side has a Lipschitz constant inversely proportional to the smallest singular value of the approximation matrix. This leads to computational difficulties as well as to a severe local Lipschitz constant in the error analysis, see \cite{KL07, KLW16}. 
The robustness of our integrator with respect to the presence of small singular values is inherited from the projector-splitting integrator, see \cite{KLW16} and could be exploited to change the rank adaptively. This would require an appropriate error monitor with respect to the choice of the approximation rank. Such strategy is feasible, but we do not address this matter here.   

It is possible to extend our approach to tensor differential equations. There, the dynamical low-rank approximation to tensors of different formats, such as tensor trains \cite{LOV15}, Tucker tensors \cite{LVW18} or hierarchical Tucker tensors \cite{LRSV13} can be applied.

The paper is structured as follows. In Section \ref{example} we illustrate the quality of the proposed approach with the help of a numerical example. In Section \ref{procedure}, we derive our method in detail. In Sections \ref{conv_theorem} and \ref{proofs} we conduct a comprehensive error analysis and give error bounds, which are independent of the norm of $A$ and of small singular values, which might appear in the approximation matrix.
Some extensions and further convergence results are given in Section \ref{extensions}.
After having presented our approach and its convergence analysis, we show how to apply the method to two essential representatives of this class of matrix differential equations: differential Lyapunov equations (DLEs) and differential Riccati equations (DREs). Finally, we illustrate our theoretical result by a numerical experiment. 
Some complementary numerical experiments can be found in \cite{MOPP17}.

\section{A motivating example}\label{example}
In this work we are interested in low-rank approximations to solutions of semilinear stiff differential equations. The integrator we propose here is a first-order method based on a splitting, which separates the stiff linear part from the non-stiff nonlinear one. The solutions of the two arising subproblems are approximated by low-rank matrices. The linear subproblem can be integrated efficiently by an exponential integrator, whereas the solution of the nonlinear differential equation is approximated by the dynamical low-rank method \cite{KL07}. The arising differential equation for the nonlinearity on the low-rank manifold is finally integrated by the projector-splitting integrator \cite{LO14}. 

The main advantage of the integration method we propose is its insensitivity against stiffness.
We illustrate this favourable behaviour with the help of an example.  

Consider the following two-dimensional partial differential equation in the variable $v(t,x,y)$
\begin{align*}
\partial_t v = \alpha\Delta v + v^3, \qquad v(0,x,y) = 16 \, x(1-x)y(1-y), 
\end{align*}
where $ \alpha  =1/50$. We solve this problem on the spatial domain $\Omega = [0,1]^2$, subject to homogeneous Dirichlet boundary conditions, for times $0 \leq t \leq T$. 
We discretize this partial differential equation in space with $m$ inner points in each direction and denote the grid size by $h$, which is $h = \frac{1}{m+1}$.
The inner grid points in the $x$ and $y$ direction are denoted by  
\[x_i = i h \quad \text{and}\quad y_j = jh \quad \text{for} \quad 1\leq i,j\leq m, \] 
respectively. The differential operator is discretized by means of second order standard finite differences. 
Denoting the one-dimensional stencil matrices in the $x$ and $y$ direction by $A_{x}$ and $A_{y}$, respectively, this results in the matrix differential equation 
\begin{align*}
\dot{U}(t) = \alpha A[U(t)] + U(t)^3, \qquad U(0) = U_0, 
\end{align*}
where $A[U(t)] = A_{x} U(t) + U(t) A_{y}$ and $U(t) \in \R^{m \times m}$. The component $U_{ij}(t)$ is the sought after approximation of $v(t,x_i,y_j)$, $1\leq i,j\leq m$. The nonlinearity is realised by an entrywise product. 

In our numerical experiment we choose $m = 500$. The reference solution is computed with DOPRI5 applied to the equivalent vector-valued equation. It is a Runge--Kutta method of order 5 with adaptive step size strategy \cite{HNW93} with high precision.
In Figure \ref{fig:pde} left, we plot the first 30 singular values of the reference solution at $T = 0.5$. 

\begin{figure}[hbt]
	\begin{minipage}[b][4.10cm][t]{.465\textwidth}
		\centering
		\includegraphics[width = \columnwidth, height= 3.9cm]{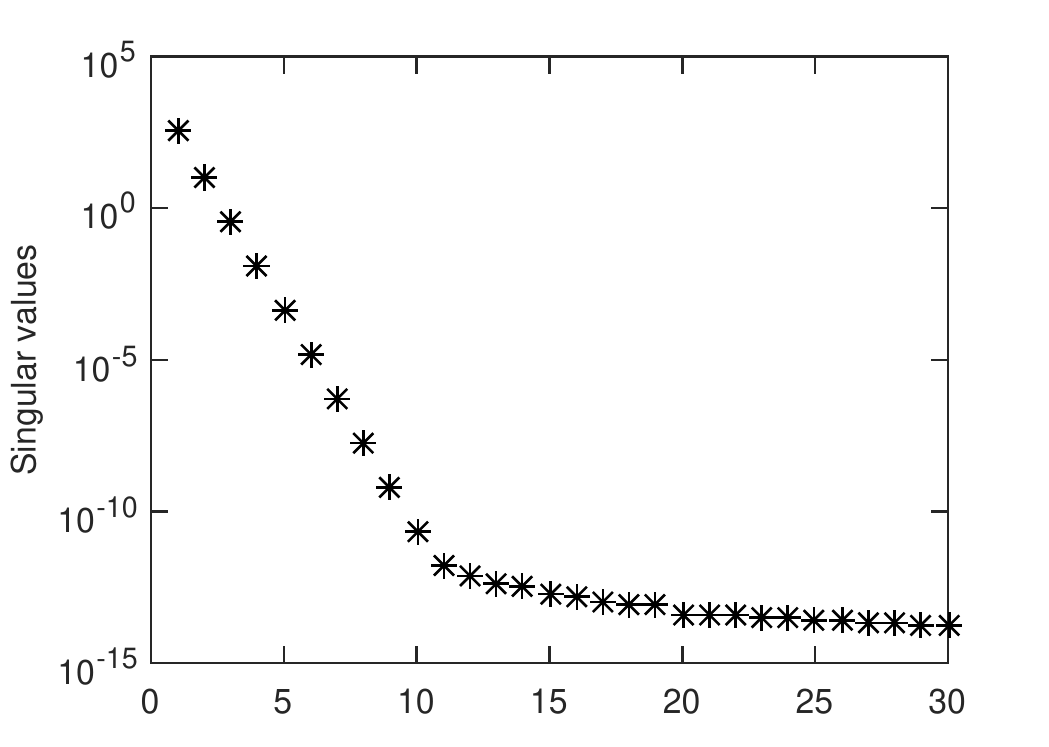}
	\end{minipage}
	\hfill
	\begin{minipage}[b][4.10cm][t]{.465\textwidth}
		\centering
		\includegraphics[width = \columnwidth, height= 4.09cm]{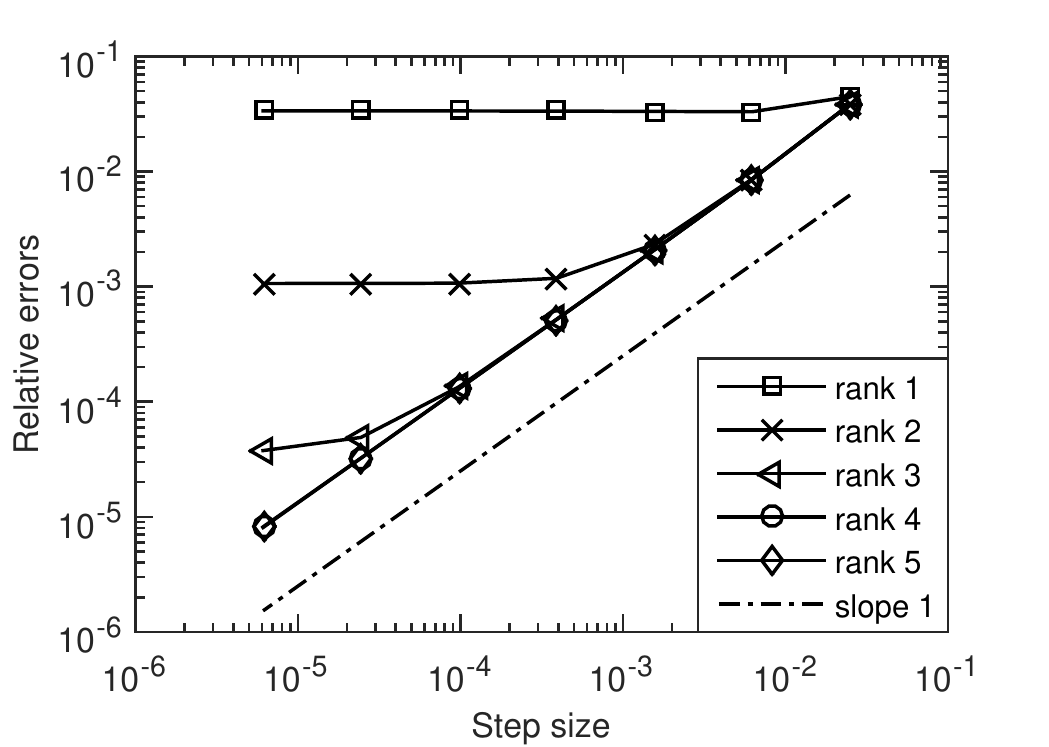}
	\end{minipage}
	\caption{Results for the solution of the considered partial differential equation at $T=0.5$. Left: First 30 singular values of the reference solution computed with DOPRI5. Right: Error of our proposed first-order splitting as a function of the step size and the approximation rank.}
	\label{fig:pde}
\end{figure}

We observe that the singular values decay quite fast. Figure \ref{fig:pde} right shows the errors of our proposed method for different approximation ranks. The error is measured in the Frobenius norm. The figure suggests an explicit dependence of the error on the rank and on the step size. If the approximation rank is chosen sufficiently large (rank 4 and 5) we solely observe the first-order error due to the splitting into the linear and the nonlinear subproblems. On the other hand, a bad choice of the approximation rank (rank 1, 2 and 3) leads to a stagnation of the error, independently on the refinement of the time step size. 

Note that standard explicit integrators would have to satisfy a CFL-like condition of the form $\tau L \leq 1$, where $L$ is the Lipschitz constant of the right-hand side of the matrix differential equation for $U(t)$ and $\tau$ is the time step size. For our choice of the parameters the time step size would be limited by the condition $\tau \leq L^{-1} \approx 2 \cdot 10^{-4}$. Our integrator, on the other hand, works fine for much larger time step sizes.

Higher dimensional partial differential equations lead to \emph{tensor differential equations} for the approximation \emph{tensor} $U(t)$. The differential operator has to be appropriately discretized, e.g., for the 3D-Laplacian the seven-point stencil tensor is obtained. We determine a low-rank solution by solving the linear part exactly and afterwards, depending on the underlying low-rank format of the approximation tensor, such as tensor trains \cite{OT09, O11} or Tucker tensors \cite{T66}, we apply the corresponding projector-splitting integrator \cite{LOV15} or \cite{LVW18}.

\section{A low-rank approximation of stiff matrix differential equations}\label{procedure}
We consider the following matrix differential equation 
\begin{align}\label{eq}
\dot{X}(t) = AX(t)+X(t)A\trasp+G(t,X(t)), \qquad X(t_0) = X^0,
\end{align}
where $X(t) \in \mathbb{C}^{m \times m}$ and $G:[t_0,\infty) \times \mathbb{C}^{m \times m} \rightarrow \mathbb{C}^{m \times m}$. The matrix $A \in \mathbb{C}^{m \times m}$ and its conjugate transpose, denoted by $A\trasp$, are time-independent. We restrict our attention here to parabolic partial differential equations; other settings are described in Section \ref{extensions}. For the moment, the matrix $A$ is typically the spatial discretization of an elliptic differential operator. Therefore, the stiffness of \eqref{eq} is induced by the matrix $A$. The nonlinearity $G$, however, is assumed to be non-stiff.
The exact full-rank solution of the above differential equation can be represented by the variation-of-constants formula as 
\begin{align*}
X(t) =\ee^{(t-t_0) A} X(t_0)\ee^{(t-t_0) A\trasp} + \int_{t_0}^t \ee^{(t-s)A} G(s,X(s))\ee^{(t-s)A\trasp} \dd s.
\end{align*} 
The aim of this work is to compute approximate solutions to $X(t)$, which are of low rank $r$ with $r \ll m$. 
We propose an integrator based on splitting methods. 

The construction of our integrator is described in the following sections. The properties of the scheme are also illustrated.

\subsection{Splitting into two subproblems}\label{split}
The structure of the matrix differential equation~\eqref{eq} motivates us to use splitting methods. For an introduction to this class of numerical integrators, we refer to \cite{HLW00} and \cite{MLQ02}. The idea behind the proposed splitting method is to benefit from the independent integration of the two arising subproblems. 

Now, splitting~\eqref{eq} into a stiff and a non-stiff part yields the following two subproblems: 
\begin{equation}\label{sub1}
\dot{X_1}(t)  = AX_1(t)+X_1(t)A\trasp, \qquad X_1(t_0) = X_1^0 
\end{equation}
and 
\begin{equation}\label{sub2}
\dot{X_2}(t) = G(t,X_2(t)), \qquad X_2(t_0) = X_2^0.
\end{equation}
We denote the solutions to the subproblems~\eqref{sub1} and \eqref{sub2} at time $t_0+\tau$ with initial values $X_1^0$ and $X_2^0$ by $\Phi_{\tau}^A(X_1^0)$ and $\Phi_{\tau}^G (X_2^0)$, respectively. Our strategy is to solve the differential equations for $X_2(t)$ and $X_1(t)$. An approximate solution of \eqref{eq} is obtained by applying the Lie--Trotter splitting scheme with step size $\tau$:
\begin{equation}\label{lie}
	\mathcal{L}_{\tau} := \PhiA \circ \PhiG.
\end{equation} 
Note, that we will refer to this scheme as \emph{full-rank Lie--Trotter splitting}. 
It results in an approximation $X^1$ of the solution $X(t)$ of~\eqref{eq} at $t = t_0 + \tau$.  Starting with $X^0 = X_2^0$, we obtain 
\begin{align*}
X^1 = \mathcal{L}_{\tau} (X^0) = (\PhiA \circ \PhiG) (X^0). 
\end{align*}
Note that the numerical solution at time $t_k = t_0+k\tau$ is $X^k = \mathcal{L}_{\tau}^k (X^0)$.
The exact solution of the homogeneous problem \eqref{sub1} is given by 
\begin{align*}
X_1(t_0+\tau) = \ee^{\tau A}X_1^0 \ee^{\tau A\trasp}.
\end{align*}
Since $X_1^0$ is typically given in low-rank factorized form (see Section \ref{proj-split} below), $X_1$ is the result of the action of a matrix exponential. Therefore, it can be efficiently computed also for large step sizes $\tau$. Methods of choice are Taylor interpolation \cite{ALH11}, interpolation at Leja points \cite{CKOR16} and Krylov subspace methods \cite{S92}. Moreover, efficient implementations on GPUs are possible, see, e.g., \cite{EO13}.

The approximate solution $X^1$ is a full-rank matrix approximation to $X(t_1)$ after one time step. Since we aim to compute rank-$r$ approximations to $X(t)$ at the time grid points, we next determine low-rank solutions of~\eqref{sub1} and~\eqref{sub2}.

\subsection{The low-rank integrator}\label{lrsol}
Denoting the manifold of rank-$r$ matrices by 
\begin{align*}
\M:= \left\{Y(t) \in \C^{m \times m}: \text{rank} \ Y(t) = r \right\},
\end{align*} 
we seek a low-rank approximation $Y \in \M$ to the solution of \eqref{eq}. In Section~\ref{split}, we have already shown how to split the differential equation into~\eqref{sub1} and~\eqref{sub2}. Now, it is the objective to determine low-rank approximations $Y_1 \in \M$ and $Y_2 \in \M$ to $X_1$ and $X_2$, respectively. To this end, we denote by $\TM$ the tangent space of the low-rank manifold $\M$ at a rank-$r$ matrix $Y$. 

We first consider the stiff subproblem \eqref{sub1}. We observe that for any $Y \in \M$, $AY+YA\trasp \in\TM$ and thus, \eqref{sub1} defines a vector field on the low-rank manifold $\M$. Hence for an initial value on the low-rank manifold $\M$, the solution of \eqref{sub1} stays in $\M$, see \cite{HM96}. This means that subproblem \eqref{sub1} is rank-preserving and so starting with a rank-$r$ initial value $Y_1^0$, the solution of 
\begin{equation} \label{Ysub1}
	\dot{Y_1}(t) = AY_1(t)+Y_1(t)A\trasp, \qquad Y_1(t_0) = Y_1^0
\end{equation}
remains of rank-$r$ for all times. 

For the second subproblem \eqref{sub2} we employ the dynamical low-rank approach \cite{KL07}. There, a rank-$r$ solution $Y_2(t)$ is determined by requiring 
\begin{align*}
\dot{Y}_2(t) \in \mathcal{T}_{Y_2(t)}\M, \qquad \lVert \dot{Y}_2(t) - \dot{X}_2(t) \rVert = \text{min},
\end{align*}
where $\mathcal{T}_{Y_2(t)}\M$ is the tangent space of the low-rank manifold $\M$ at the current approximation $Y_2(t)$. The above condition is equivalent to orthogonally projecting the right-hand side of \eqref{sub2} onto the tangent space $\mathcal{T}_{Y_2(t)}\mathcal{M}$. This results in an evolution equation for $Y_2(t)$, which is of the form
\begin{align}\label{Ysub2}
	\dot{Y}_2(t) = P(Y_2(t))G(t,Y_2(t)), \qquad Y_2(t_0) = Y_2^0,
\end{align}
where the initial value $Y_2^0$ is on the low-rank manifold $\M$. The orthogonal projection is denoted by $P$. 
This differential equation needs to be solved numerically. 
A favorable integration scheme is the so-called projector-splitting integrator \cite{LO14} which will be described in detail in Subsection \ref{proj-split}. 
The authors of \cite{KLW16} have proved that this integrator is robust with respect to the presence of small singular values. This is a crucial property, since in most applications the rank might not be known in advance. For accuracy reasons the rank is often over-approximated and small singular values enter in the approximation matrices. 
Solving the system of differential equations for the low-rank factors of the solution, as proposed in \cite{KL07}, becomes cumbersome. Standard integrators such as explicit and implicit Runge--Kutta methods suffer from the possible ill-conditioning of the arising matrices. For further details we refer to the discussions in \cite{LO14,KLW16}. 

After having applied the projector-splitting integrator to \eqref{Ysub2}, the resulting low-rank approximation of $X_2(t)$ at $t_0+\tau$ is
\begin{align*}
 Y_2^1 = \phiG(Y_2^0),
\end{align*}
where $\phiG$ denotes the approximated solution operator of \eqref{Ysub2}.
In a nutshell, the integration method we propose consists of first splitting the matrix differential equation \eqref{eq} and then approximating the subproblems \eqref{sub1} and \eqref{sub2} with respect to low-rank. Hence combining the flow $\phiG$ of the low-rank solution of \eqref{sub2} with the exact flow $\PhiA$ of \eqref{sub1}, which is of low rank when starting with low-rank initial data, yields the desired approximation matrix $Y(t)$. We call this procedure \emph{low-rank Lie--Trotter splitting} and denote it by  
\begin{equation} \label{approx_lie}
	\LL_{\tau} := \PhiA \circ \phiG. 
\end{equation}
Thus, starting with $Y^0 = Y_2^0$, we obtain the rank-$r$ approximation of the solution of \eqref{eq} at time $t_0+\tau$, i.e.,  
\begin{equation} \label{method}
Y^1 = \LL_{\tau}(Y^0) = (\PhiA \circ \phiG) (Y^0),  
\end{equation}
where we assume $Y^0$ to be a rank-$r$ approximation to $X^0$. At $t_k = t_0+k\tau$ we obtain $Y^k = \mathcal{I}^k_{\tau} (Y^0)$.

\subsection{The projector-splitting integrator}\label{proj-split}
The low-rank approximation $Y_2$ is not computed directly from \eqref{Ysub2} by applying a standard integration method, but by an efficient integrator, which benefits from the underlying low-rank format for matrices.  
In the following, we drop the subscript in $Y_2$ and describe the integrator for any $Y \in \M$ and any problem of the form
\begin{align*}
\dot{Y}(t) = P(Y(t))G(t,Y(t)), \qquad Y(t_0) = Y^0 \in \M.
\end{align*}

The projector-splitting integrator introduced in \cite{LO14} is based on the observation that every rank-$r$ matrix $Y(t) \in \C^{m \times m}$ can be represented as 
\begin{align*}
	Y(t) = U(t)S(t)V(t)^*,
\end{align*}
where $U(t), V(t) \in \C^{m \times r}$ have orthonormal columns. The square matrix $S(t) \in \C^{r \times r}$ is invertible and has the same non-zero singular values as $Y(t)$. In contrast to the singular value decomposition (SVD), this non-unique factorization does not require $S(t)$ to be diagonal. From the computational perspective, this representation has the advantage of a significant reduction in memory requirements and computational cost if $r \ll m$.
 
The projector-splitting integrator makes use of this SVD-like factorization, in the sense that the time integration is performed only on the low-rank factors. 
It is based on splitting the projection $P(Y)$ onto the tangent space $\mathcal{T}_Y\M$ of the low-rank manifold $\M$. Following \cite[Lemma 4.1]{KL07}, the orthogonal projection $P(Y)$ at the current approximation matrix $Y = USV\trasp \in \M$ can be written as  
\begin{equation} \label{projection}
\begin{split}
P(Y) G(t,Y) &\ = UU\trasp G(t,Y) - UU\trasp G(t,Y) VV\trasp + G(t,Y) VV\trasp \\
&\  =: P^a(Y)G(t,Y) -P^b(Y)G(t,Y) + P^c(Y)G(t,Y).
\end{split}
\end{equation}
Further, $UU\trasp$ and $VV\trasp$ are orthogonal projections onto the spaces spanned by the range and co-range of $Y$, respectively.
One time step from $t_0 \to t_1=t_0+\tau$ of the first-order integrator consists of solving the evolution equations 
\begin{alignat*}{2}
\dot{Y}^a(t) &= P^a(Y)G(t,Y),  \qquad & Y^a(t_0) &= Y^0, \\
\dot{Y}^b(t) &= -P^b(Y)G(t,Y),  & Y^b(t_0) & = Y^a(t_1), \\
\dot{Y}^c(t) &= P^c(Y)G(t,Y),  & Y^c(t_0) & = Y^b(t_1)
\end{alignat*}
consecutively, where $Y^c(t_1)$ is the approximate solution to $Y(t_1)$. In practice, those differential equations have to be solved approximately using a numerical method, e.g., a Runge--Kutta method.
Higher-order methods can be obtained from the first-order scheme by employing the standard composition techniques, see \cite{LO14, KLW16}.

\section{The main convergence result} \label{conv_theorem}

In this section we describe the framework in which the convergence proof can be carried out and formulate the main convergence result. Further, we give an outline of the proof. The technical details are postponed to Section \ref{proofs}.
It is worth remarking that the convergence analysis of the low-rank Lie--Trotter splitting \eqref{approx_lie} is performed without introducing Lipschitz conditions of the full right-hand side of \eqref{eq} nor of the stiff subproblem \eqref{sub1}.  

Let us consider the Hilbert space $\C^{m\times m}$, endowed with the Frobenius norm $\lVert \cdot \rVert$. 
Let $A \in \C^{m\times m}$ and $G: [t_0,T] \times \mathbb{C}^{m \times m} \rightarrow \mathbb{C}^{m \times m}$.
In the following, we are given an initial data $X^0$ and a final integration time $T$ such that the matrix differential equation~\eqref{eq} has a solution $X(t)$ for $t_0\leq t\leq T$. 
We assume that, given a rank-$r$ approximation $Y^0$ of the initial value $X^0$ such that 
\begin{align*}
\lVert X^0 - Y^0 \rVert \leq \delta
\end{align*}
for some $\delta \geq 0$, the exact rank-$r$ solution
\begin{equation*}
Y(t) =\ee^{(t-t_0) A} Y^0 \ee^{(t-t_0) A\trasp} + \int_{t_0}^t \ee^{(t-s)A} P(Y(s))G(s,Y(s))\ee^{(t-s)A\trasp} \dd s
\end{equation*}
of the matrix differential equation \eqref{eq} exists for $t_0\leq t\leq T$. 

For proving convergence, we further need the following assumption:
\begin{hp} \label{hp_split}
	We assume that the following properties hold.
	\begin{enum_ass}
		\item  \label{A_ass}
		There exists $\omega \in \R$ and $C_s > 0$, such that the matrix $A$ satisfies
		\begin{align} 
		\label{bound} 
		\lVert \ee^{tA} Z \ee^{tA\trasp} \rVert & \leq \ee^{t\omega} \lVert Z \rVert,  \\ 
		\label{smooth_pr}
		\lVert \ee^{tA} (AZ+ZA\trasp) \ee^{tA\trasp} \rVert &\leq \frac{1}{t}C_s \ee^{t\omega}\lVert Z \rVert 
		\end{align}
		for all $t>0$ and all $Z \in \C^{m \times m}$.
		\item \label{G_ass} $G$ is continuously differentiable in a neighbourhood of the exact solution. 
	   	\item \label{M-R}
	   	There exists $\varepsilon >0$ such that for all $t_0\leq t\leq T$ 
	   	\begin{align*}
	    G(t,Y(t)) = M(t,Y(t)) + R(t,Y(t)),
		\end{align*}
		where $M(t,Y(t)) \in \mathcal{T}_{Y(t)} \M$ and $\lVert R(t,Y(t)) \rVert \leq \varepsilon$.
	\end{enum_ass}
\end{hp}

The above assumptions require some explanations and discussion. Moreover, we need to specify some crucial properties for the proof of the error bounds given in Theorem~\ref{main_th}.
	\begin{enum_ass} 
		\item 
		Matrix differential equations of the form \eqref{eq} are typically stemming from parabolic partial differential equations. We refer to, e.g., the example in Section \ref{example} and to the discussion about differential Lyapunov and differential Riccati equations in Sections \ref{sect:DLE} and \ref{sect:DRE}, respectively. 
		The matrix operator $F$, given by 
		\[ F(X) = AX+XA\trasp, \quad X\in \C^{m\times m} \] is equivalent to the operator $\mathcal{F}$ 
		\[\mathcal{F} (x) = \mathcal{A} x = (I_m \otimes A + A \otimes I_m ) x, \quad x =\text{vec}(X) \in \C^{m^2}, \]
		where we denote by $\otimes$ the Kronecker product and by $\text{vec}(\cdot)$ the columnwise vectorization of a matrix into a column vector. 
		Then, the bounds \eqref{bound} and \eqref{smooth_pr} can be translated using the vector 2-norm $\lVert \cdot \rVert_2$ as
		\begin{align*}
			\lVert \ee^{t\mathcal{A}} z \rVert_2 \leq \ee^{t\omega} \lVert z \rVert_2,  \\ 
			\lVert \ee^{t\mathcal{A}} \mathcal{A} z \rVert_2 \leq \frac{1}{t}C_s \ee^{t\omega}\lVert z \rVert_2, 
		\end{align*}
		for all $t>0$ and all $z = \text{vec}(Z) \in \C^{m^2}$.
		These properties are well known in the context of semigroup theory for strongly elliptic operators, see, e.g., \cite{EN06}, \cite{P83}. In particular, the Lumer-Phillips theorem \cite[Sect.$\,$12]{ReRo93} provides a practical criterion for generators of quasicontraction semigroups in Hilbert spaces. This theorem also applies to standard space discretizations of such operators and in particular shows that the constants $\omega$ and $C_s$ can be chosen independently of $m$. In particular, they are independent of the problem's stiffness.
		
		\item As a consequence of Assumption \ref{hp_split}\ref{G_ass}, the function $G$ is locally Lipschitz continuous with constant $L$, and $G$ is bounded by $B$ in a neighbourhood of the solution $X(t)$, i.e.,
		\begin{align}\label{LB}
		\begin{split}
		\lVert G(t,\widehat{X})-G(t,\widetilde{X}) \rVert &\leq L \lVert \widehat{X} -\widetilde{X} \rVert, \\
		\lVert G(t,\bar{X}) \rVert &\leq B,
		\end{split}
		\end{align}
		as long as  $\lVert \widehat{X}-X(t)\rVert \leq \gamma$, $\lVert \widetilde{X}-X(t)\rVert \leq \gamma$, and $\lVert \bar{X}-X(t)\rVert \leq \gamma$ for $t_0 \leq t \leq T$ for given $\gamma > 0$.
		The constants $L$ and $B$ depend on $\gamma$. 
		\item We assume that $G(t,Y)$ consists of a tangential part $M(t,Y)$ and a small perturbation term $R(t,Y)$. This means that $G$, when evaluated along the low-rank solution, is in the tangent space up to a small remainder of size $\varepsilon$. This assumption is crucial in order to have a good low-rank approximation, since if the remainder is large, low-rank approximation is inappropriate. 
	\end{enum_ass}

Having clarified the assumption, we are now in a position to state the main result of this paper.

\begin{theorem}[Global error of the low-rank Lie--Trotter splitting integrator]\label{main_th}
	Under Assumption \ref{hp_split}, there exists $\tau_0$ such that for all step sizes $0 < \tau \leq \tau_0$ the error of the low-rank Lie--Trotter splitting integrator \eqref{approx_lie} is uniformly bounded on $t_0\leq t_0+n\tau \leq T$ by 
	\begin{align*}      
	\lVert X(t_0+n\tau) - \LL_{\tau}^n (Y^0)\rVert \leq c_0 \tau (1+\lvert \log \tau \rvert) + c_1 \delta + c_2 \varepsilon, 
	\end{align*}
	where $c_0$, $c_1$ and $c_2$ depend on $\omega$, $C_s$, $L$, $B$ and $T$, but are independent of $\tau$ and $n$.  
\end{theorem}
Note that $\tau_0$ depends only on the size of the Lipschitz constant of $G$. 

In order to facilitate the analysis of \eqref{approx_lie}, we study the global error by introducing auxiliary quantities. The construction of the method already suggests that the global error is composed by the following terms: 
\begin{enumerate}[label=(\roman*)]
	\item The global error of the full-rank Lie--Trotter splitting \eqref{lie}, applied to \eqref{sub1} and \eqref{sub2}: 
	\begin{align*}  
	E^n_{sp} = X(t_0+n\tau) - (\PhiA \circ \PhiG)^n(X^0).
	\end{align*}
	\item The propagation of the difference between the full-rank initial data $X_0$ and its low-rank approximation $Y_0$ by the full-rank Lie--Trotter splitting \eqref{lie}: 
	\begin{align*}
	E^n_{\delta} =  (\PhiA \circ \PhiG)^n(X^0) - (\PhiA \circ \PhiG)^n(Y^0).
	\end{align*}
	\item The difference between the full-rank Lie--Trotter splitting \eqref{lie} and the low-rank Lie--Trotter splitting \eqref{approx_lie} applied to $Y_0$: 
	\begin{align*}
	E^n_{lr} = (\PhiA \circ \PhiG)^n (Y^0) - (\PhiA \circ \phiG)^n (Y^0). 
	\end{align*} 
\end{enumerate}
Hence, the global error in Theorem \ref{main_th} is obtained as the sum of $E^n_{sp}$, $E^n_{\delta}$ and $E^n_{lr}$ as illustrated in Figure \ref{fig:Windermere}. Those three contributions are studied in detail in the following section. 
	 
\begin{figure}
	\centering
	\includegraphics[trim = 80mm 30mm 120mm 60mm, width=.3\textwidth]{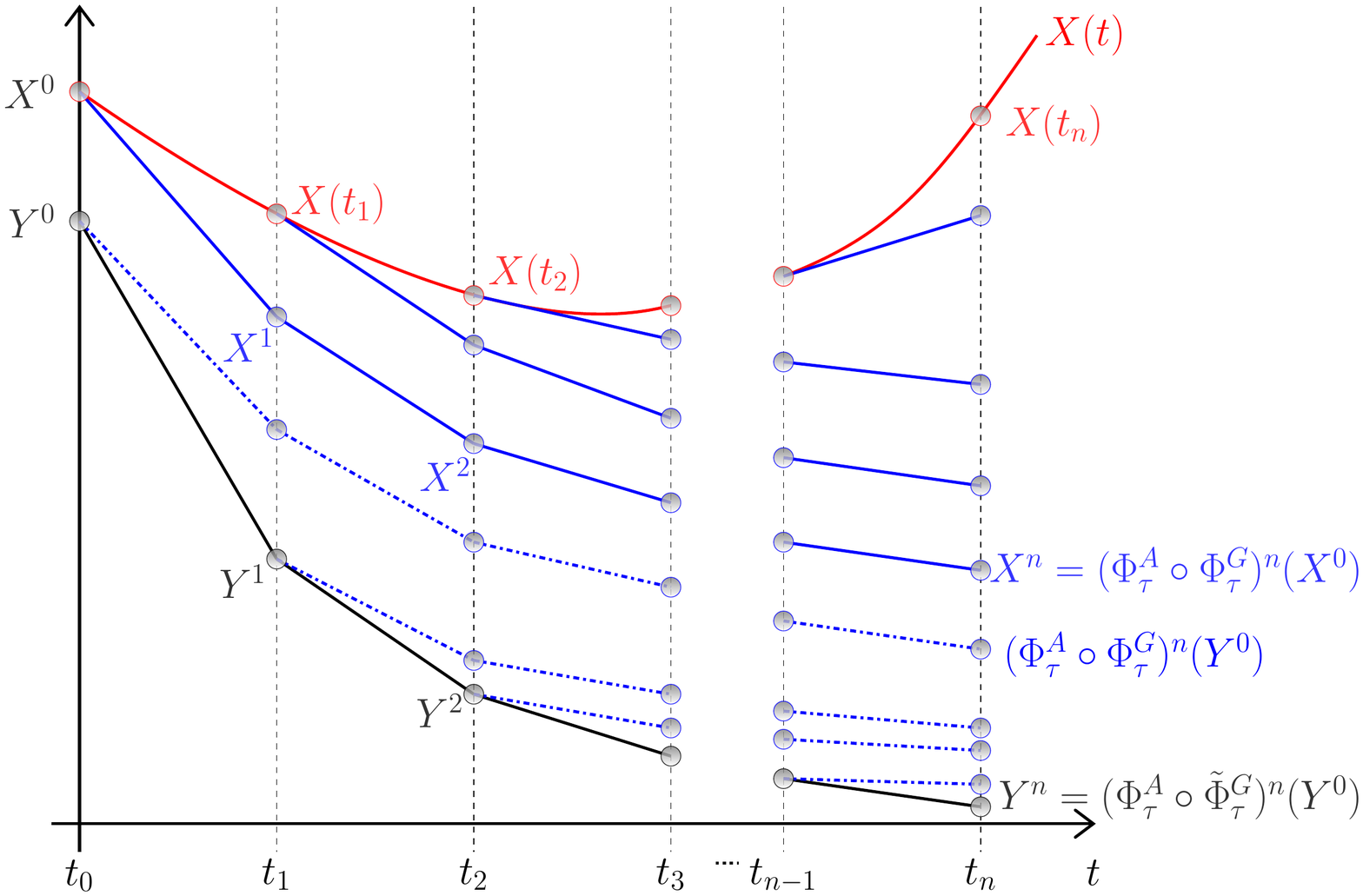}
	\caption{Schematic illustration of the convergence analysis. The uppermost curve (in red) depicts the exact solution $X(t)$ of \eqref{eq}, whereas the lowermost (in black) shows the solution obtained by the low-rank Lie--Trotter splitting \eqref{approx_lie}. All other lines (in blue) represent the auxiliary values obtained by the application of the full-rank Lie--Trotter splitting \eqref{lie} either to a full-rank initial data (continuous lines) or to a low-rank initial data (dash-dotted lines).}
	\label{fig:Windermere}
\end{figure}

\section{Detailed convergence analysis} \label{proofs}

The aim of this section is to provide a convergence analysis of the low-rank Lie--Trotter splitting \eqref{approx_lie}. We give a detailed proof of Theorem \ref{main_th} and in particular we state and prove error bounds for the three contributions listed above. First, we prove the error bound of the full-rank Lie--Trotter splitting in Subsection \ref{err_spl}, followed by the error estimate for the low-rank Lie--Trotter splitting in Subsection \ref{err_lr}. The propagation of the difference between the full and low-rank initial data requires just the stability of the full-rank Lie--Trotter splitting. This is shown in Subsection \ref{final_proof}.

Note that our proofs rely on the constants $L$ and $B$ in Assumption~\ref{hp_split}\ref{G_ass}. In order to bound these constants, we have to ensure that the numerical approximations stay in a fixed compact neighbourhood $\mathcal{U}$ of the exact solution. This follows (by recursion) from the given proofs, taking into account that the arising constants can be controlled in terms of $L$, $B$ and the final time $T$. An appropriate choice of the maximum step size $\tau_0$ finally guarantees that all considered approximations stay in $\mathcal{U}$. 

\subsection{The error of the full-rank Lie--Trotter splitting} \label{err_spl}

The convergence of the full-rank splitting scheme \eqref{lie} is stated in the following theorem. 
The ideas in the proof can be traced back to, e.g., \cite{EO15} and \cite{JL00}.

\begin{prop}[Global error of the full-rank Lie--Trotter splitting]\label{spl_err} \label{glob_err_prop_sp}
	Under Assumption~\ref{hp_split}, the full-rank Lie--Trotter splitting \eqref{lie} is first-order convergent, i.e., the error bound
	\[ \lVert X(t_0+n\tau) - (\PhiA \circ \PhiG)^n(X^0) \rVert \leq C \tau (1+\lvert \log \tau \rvert) \]
	holds uniformly on $t_0\leq t_0+n\tau \leq T$. The constant $C$ depends on $\omega$, $C_s$, $L$, $B$ and $T$, but is independent of $\tau$ and $n$. 
\end{prop}

\begin{proof}
	The solution of the matrix differential equation \eqref{eq} can be expressed by means of the variation-of-constants formula. Given the initial value $X(t_{k-1})=Z$, the solution at time $t_k = t_{k-1}+\tau$ with step size $\tau > 0$ is 
	\begin{align*}
	X(t_k) =\ee^{\tau A} Z\ee^{\tau A\trasp} + \int_0^\tau\ee^{(\tau-s)A} G(t_{k-1}+s,X(t_{k-1}+s))\ee^{(\tau-s)A\trasp}\dd s.
	\end{align*}
	The exact solution of the first full-rank subproblem \eqref{sub1} at $t_k$ with initial value $X_1(t_{k-1}) = X_2(t_k)$ is given by
	\begin{equation} \label{exact_1} 
	\Phi^A_{\tau}(X_2(t_k)) = X_1(t_k) =\ee^{\tau A} X_2(t_k) \ee^{\tau A\trasp},
	\end{equation} 
    whereas the exact solution of the second full-rank subproblem \eqref{sub2} with initial value $X_2(t_{k-1})=Z$ can be expressed as 
	\begin{equation} \label{exact_2} 
	\Phi^G_{\tau}(Z) = X_2(t_k) = Z+\tau G(t_{k-1},Z)+\int_{0}^{\tau} (\tau-s) \ddot{X}_2 (t_{k-1}+s) \dd s.
	\end{equation}
	Composing \eqref{exact_1} with \eqref{exact_2} gives the full-rank Lie--Trotter splitting solution    
	\begin{equation} \label{exact_full}
	\mathcal{L}_{\tau} (Z) = \ee^{\tau A} Z \ee^{\tau A\trasp} + \tau \ee^{\tau A} G(t_{k-1},Z) \ee^{\tau A\trasp} + \int_{0}^{\tau} (\tau-s ) \ee^{\tau A} \ddot{X}_2(t_{k-1}+s) \ee^{\tau A\trasp}\dd s.
	\end{equation}
	The local error of the method at $t_k$ is  
	\begin{align*}
	e^{k}_{sp} &= X(t_k) - \mathcal{L}_{\tau} (X(t_{k-1})) \\
	 &= \int_0^\tau\ee^{(\tau-s)A} G(t_{k-1}+s,X(t_{k-1}+s))\ee^{(\tau-s)A\trasp} \dd s  \\
	& - \tau \ee^{\tau A} G(t_{k-1},X(t_{k-1})) \ee^{\tau A\trasp} - \int_{0}^{\tau} (\tau-s ) \ee^{\tau A} \ddot{X}_2(t_{k-1}+s) \ee^{\tau A\trasp}\dd s.
	\end{align*}
	Let $f(s)=\ee^{(\tau-s)A} G(t_{k-1}+s,X(t_{k-1}+s))\ee^{(\tau-s)A\trasp}$. Then the first integral above can be rewritten as 
	\[ 
	\int_0^{\tau}{f(s) \dd s} = \int_0^{\tau} \left[ f(0) + \int_0^{s} \dot{f} (r) \dd r \right]\dd s. 
	\]
	Using the fact that a matrix commutes with its exponential, the derivative of $f$ is 
	\[ 
	\dot{f}(s) = -\ee^{(\tau-s)A} \left(AG+GA\trasp - \frac{\dd{G}}{\dd s}\right) \ee^{(\tau-s)A\trasp}.
	\]
	Recall that the function $G$ is assumed to be continuously differentiable in a neighbourhood of the exact solution. Hence, employing the boundedness of
	\[ 
    \ddot{X}_2(t) = \frac{\dd{G}}{\dd t}(t) = \partial_t G(t,X_2(t))+\partial_X G(t,X_2(t))G(t,X_2(t))
	\]
	we are left with a simpler form of the local error:
	\begin{equation} \label{loc_err_formula_sp} 
	e^{k}_{sp} = -\int_0^{\tau} \int_0^{s} \ee^{(\tau-r)A} \left(AG+GA\trasp \right) \ee^{(\tau-r)A\trasp} \dd r \dd s + \mathcal{O}(\tau^2).  
	\end{equation}
	Due to the presence of the matrix $A$, we do not bound the local error \eqref{loc_err_formula_sp} directly. Instead, we solve the error recursion first.
	Recalling that $X^{n-1}~=~\mathcal{L}_{\tau}^{n-1}(X^0)$, we write the global error of the Lie--Trotter splitting as 
	\begin{align*} 
	E^n_{sp} = \mathcal{L}_{\tau} (X(t_{n-1}))-\mathcal{L}_{\tau} (X^{n-1})+e_{sp}^n,
	\end{align*}
	where the first two terms represent the propagation of $E^{n-1}_{sp}$ by the numerical method $\mathcal{L}_{\tau}$, which is nonlinear.
	Making use of formula \eqref{exact_full}, we write  
	\begin{align} \label{lie_prop} 
	\mathcal{L}_{\tau} (X(t_{n-1}))-\mathcal{L}_{\tau} (X^{n-1}) = \ee^{\tau A} E^{n-1}_{sp}\ee^{\tau A\trasp} + \ee^{\tau A} H(X(t_{n-1}),X^{n-1})\ee^{\tau A\trasp},
	\end{align}
	where 
	\begin{align*} 
	H(X(t_{n-1}),X^{n-1})  = & \,\tau \left[G(t_{n-1},X(t_{n-1}))-G(t_{n-1},X^{n-1}) \right] \\
	& + \int_{0}^{\tau} (\tau-s) \left[ \ddot{X}_2(t_{n-1}+s) - \ddot{\widetilde{X}}_2 (t_{n-1}+s) \right] \dd s.
	\end{align*}
	The functions $X_2$ and $\widetilde{X}_2$ are the solutions of the second subproblem \eqref{sub2} with initial values $X(t_{n-1})$ and $X^{n-1}$, respectively. 
	Starting from $X(t_0)$ and $X^0$ and using expression \eqref{lie_prop} for their propagation by the Lie--Trotter splitting method, we rewrite the global error as
	\begin{align} \label{glob_err_formula_sp}
	\begin{split}
	E^n_{sp} = & \, \ee^{n \tau A} E^0_{sp} \ee^{n\tau A\trasp} + \underbrace{\sum_{k = 0}^{n-1} \ee^{(n-k)\tau A} H(X(t_{k}),X^{k})\ee^{(n-k)\tau A\trasp}}_{=: D_1} \\
	& + \underbrace{\sum_{k=1}^{n}\ee^{(n-k)\tau A} e^k_{sp} \ee^{(n-k)\tau A\trasp}}_{=: D_2}.
	\end{split}
	\end{align}
	By the choice of the initial value $X(t_0) = X^0$ we have $ \lVert E^0_{sp} \rVert = 0$. Since the expression $H$ mainly consists of the nonlinear function $G$, which by Assumption~\ref{hp_split}\ref{G_ass} is Lipschitz continuous, and of its derivative $\ddot{X}_2$, which is continuous, we have the bound 
	\begin{align*}
	\lVert H(X(t_k),X^k) \rVert \leq C (\tau \lVert E^k_{sp} \rVert + \tau^2). 
	\end{align*}
	Hence, property \eqref{bound} yields the following bound for the second term in the representation of the global error \eqref{glob_err_formula_sp}: 
	\begin{equation} \label{D1} 
	\lVert D_1 \rVert \leq C \sum_{k = 0}^{n-1} \ee^{(n-k)\tau\omega} \left(\tau \lVert E^k_{sp} \rVert + \tau^2 \right) \leq C \tau  \left(\sum_{k = 0}^{n-1} \lVert E^k_{sp} \rVert + 1\right). 
	\end{equation}
	Now, in order to bound $D_2$, we have to consider the form of the local error in \eqref{loc_err_formula_sp}. We have 
	\begin{align*}
	\lVert D_2 \rVert & \leq  \sum_{k=1}^{n} \bigg\lVert \int_{0}^{\tau} \int_{0}^{s} \ee^{(\tau-r)A}  \left( \ee^{(n-k)\tau A} (A G + G A\trasp ) \ee^{(n-k)\tau A\trasp} \right)  \ee^{(\tau-r)A\trasp} \dd r \dd s \bigg\rVert. 
	\end{align*}
	The quantity in the parentheses can be bounded by means of assumption \eqref{smooth_pr}. Further, employing assumption \eqref{bound} we obtain
	\begin{align*}
	\lVert D_2\rVert & \leq  C \sum_{k=1}^{n-1} \frac{1}{(n-k)\tau}\ee^{(n-k)\tau \omega} \int_{0}^{\tau} \int_{0}^{s} \ee^{(\tau-r)\omega} \dd r \dd s + C\tau
	\end{align*}
	and achieve the following bound
	\begin{equation} \label{D2}
	\lVert D_2 \rVert \leq C \tau^2 \sum_{k=1}^{n-1} \frac{1}{k \tau} + C \tau.
	\end{equation}
	Now, collecting \eqref{glob_err_formula_sp}, \eqref{D1} and \eqref{D2} yields the error bound:
	\[ 
	\lVert E^k_{sp} \rVert \leq  C \tau \sum_{k = 0}^{n-1} \lVert E^k_{sp} \rVert + C \tau \log n  + C \tau. 
	\]
	The global error bound follows now from a discrete Gronwall inequality, see, e.g., \cite{Gronwall86}.
\end{proof}

\subsection{The low-rank Lie--Trotter splitting} \label{err_lr}
In this section we compare the full-rank Lie--Trotter splitting \eqref{lie} and the low-rank Lie--Trotter splitting. We recall that the solution obtained with the latter is given by $Y^n=(\PhiA \circ \phiG)^n (Y^0)$. The following proposition states the error bound. Its proof is given at the end of this section.
\begin{prop}\label{globerr}
	Under Assumption \ref{hp_split}, the difference $E^n_{lr}=(\PhiA \circ \PhiG)^n (Y^0) - (\PhiA \circ \phiG)^n (Y^0)$ is uniformly bounded on $t_0 \leq t_0+n\tau \leq T$ as 
	\begin{align*}
	\lVert E^n_{lr} \rVert \leq c_2 \varepsilon + c_3 \tau , 
	\end{align*}
	where the constants $c_2$ and $c_3$ depend on $\omega$, $L$, $B$ and $T$, but are independent of $\tau$ and $n$.  
\end{prop}
The low-rank Lie--Trotter splitting defined in \eqref{approx_lie} with initial value $Y^0$ results, after one time step, in $Y^1=(\PhiA \circ \phiG) (Y^0)$.
It consists of first applying the projector-splitting integrator to the evolution equation \eqref{Ysub2} for the nonlinearity on the tangent space $\mathcal{T}_{Y_2}\M$ and then solving exactly the first subproblem \eqref{Ysub1} with initial value $\phiG (Y^0)$.
We start with the following preliminary result. 

\begin{lemma} \label{locerr}
	Under Assumption \ref{hp_split}, the following bound holds uniformly for each $n \geq 1$ satisfying $t_0 \leq t_0 + n\tau \leq T$ 
	\begin{align*}
	\lVert  (\PhiA \circ \PhiG) (Y^{n-1}) - (\PhiA \circ \phiG) (Y^{n-1}) \rVert \leq b_1\varepsilon \tau + b_2 \tau^2,  
	\end{align*}
	as long as $\lVert Y^{n-1}-Y(t_{n-1})\rVert \leq \gamma$ for given $\gamma > 0$, see \eqref{LB}. The constants $b_1$ and $b_2$ depend on $\omega$, $L$, $B$ and $T$, but are independent of $\tau$ and $n$.
\end{lemma}	

\begin{proof}
	We observe that 
	\begin{align*}
	\lVert  (\PhiA \circ \PhiG) (Y^{n-1}) - (\PhiA \circ \phiG) (Y^{n-1}) \rVert & = \lVert  (\PhiA \circ (\PhiG - \phiG)) (Y^{n-1}) \rVert \\
	& \leq \ee^{\tau \omega} \lVert (\PhiG - \phiG) (Y^{n-1}) \rVert,  
	\end{align*}	
	where in the last step we employ bound \eqref{bound} for the matrix exponential operator.
    For estimating the remaining local error $(\PhiG - \phiG) (Y^{n-1})$ of the projector-splitting integrator, we mainly refer to the error analysis in \cite{KLW16}. Let us consider the nonlinear subproblem \eqref{Ysub2}, which by Assumption \ref{hp_split}\ref{M-R} can be written as 
    \begin{align*}
    \dot{Y}_2(t) = M(t,Y_2(t)) + P(Y_2(t))R(t,Y_2(t)), \quad Y_2(t_{n-1}) = Y^{n-1},
    \end{align*}
    for all $n \geq 1$ satisfying $t_0 \leq t_0+n \tau \leq T$.
    Now, dropping the perturbation term yields the auxiliary problem
    \begin{align*}
    \dot{W}(t) = M(t,W(t)), \quad W(t_{n-1}) = W^{n-1}.
    \end{align*}
    Following \cite[Lemma 2.2]{KLW16}, there exists $W^{n-1}$ such that $\lVert Y^{n-1} - W^{n-1} \rVert \leq \tau(4BL\tau + 2\varepsilon)$ and the following bound holds 
    \begin{align*}
    \lVert \phiG(Y^{n-1}) - W(t_n) \rVert \leq \tau (9BL\tau + 4\varepsilon).
    \end{align*} 
    Moreover, by the bound of the perturbation term $R$ and the Lipschitz constant of $G$, we obtain by a Gronwall inequality 
    \begin{align*}
    \lVert \Phi_{\tau}^G(Y^{n-1}) - W(t_n) \rVert \leq \ee^{L\tau}(\tau(4BL\tau + 2\varepsilon)+\tau \varepsilon).
    \end{align*}
    Collecting those two error estimates results in the local error
	\begin{align*}
	\lVert  (\PhiG - \phiG) (Y^{n-1}) \rVert \leq (4BL\ee^{L\tau} + 9BL)\tau^2 + (3\ee^{L\tau} + 4)\varepsilon \tau, 
	\end{align*}
    which proves the stated local error bound for $b_1 = \ee^{\tau \omega}(3\ee^{L\tau} + 4) $ and $b_2 =\ee^{\tau \omega}(4BL\ee^{L\tau} + 9BL)$.
\end{proof}

With this local error estimate at hand, we are now in a position to prove the bound for $E^n_{lr}$. 
\begin{proof}[Proof of Proposition \ref{globerr}]
Let $\widehat{Y}$, $\widetilde{Y} \in \M$. Employing bound \eqref{bound} and the Lip\-schitz continuity \eqref{LB} of the nonlinearity $G$, we obtain 
\begin{align*}
\begin{split}
\lVert ( \PhiA \circ \PhiG) (\widehat{Y}) - ( \PhiA \circ \PhiG ) (\widetilde{Y}) \rVert & = \lVert \PhiA (\PhiG (\widehat{Y})  - \PhiG(\widetilde{Y})) \rVert \\
&\leq \ee^{(L+\omega)\tau} \lVert \widehat{Y} - \widetilde{Y} \rVert,
\end{split}
\end{align*}
which shows stability of the splitting method $\mathcal{L}_{\tau} = \PhiA \circ \PhiG$. 
Combining the stability with the result of Lemma \ref{locerr}, we obtain the following error recursion
\[ \lVert E_{lr}^n \rVert \leq b_1 \epsilon \tau + b_2 \tau^2 + \ee^{(L + \omega)\tau} \lVert E_{lr}^{n-1} \rVert. \] 
The stated bound follows by standard arguments.
\end{proof}

\subsection{Proof of Theorem \ref{main_th}} \label{final_proof}
Finally, we are in a position to combine the results of the previous sections and prove the main result of this paper. 
\begin{proof}[Proof of Theorem \ref{main_th}]
What remains is to give a bound for the propagation $E^n_{\delta}$ of the initial error $\lVert X^0 - Y^0 \rVert$ by $\mathcal{L}_{\tau}$.
Due to stability of $\mathcal{L}_{\tau}$, we have the following bound
\[ \lVert (\PhiA \circ \PhiG )^n (X^0) - ( \PhiA \circ \PhiG )^n (Y^0) \rVert \leq \ee^{(L+\omega)(T-t_0)} \lVert X^0-Y^0 \rVert. 
\] 
Combining the three components of the global error results in the stated bound with $c_0$ containing $C$ and $c_3$, which come from Proposition~\ref{spl_err} and from Proposition \ref{globerr}, respectively, with $c_1$, which is the constant of the bound for the propagated initial value, and with $c_2$, which appears in Proposition \ref{globerr}.
\end{proof}

As a remark, we point out that the low-rank Lie--Trotter splitting integrator is not sensitive to the presence of small singular values.
The low-rank solution of the linear problem is computed directly by exponential integrators, where possibly appearing small singular values do not cause difficulties. Further, they can also occur in the approximation matrix of the nonlinear subproblem. But since we are applying the projector-splitting integrator, which is robust with respect to small singular values, our integration method inherits this favorable property.

\section{Extensions and further convergence results} \label{extensions}
In this section, we comment on the possible extension of the low-rank Lie--Trotter splitting~\eqref{approx_lie} to a \emph{low-rank Strang splitting}, and we sketch some other situations in which the convergence proof of Section~\ref{proofs} also holds.

\subsection{The low-rank Strang splitting}
The main drawback of the Lie--Trotter splitting scheme in application is its low order. Composing the scheme with its adjoint method, which is again a Lie--Trotter splitting with the order of flows reversed, one obtains the formally second-order Strang splitting. In the low-rank situation, the resulting scheme is given by
$$
\Phi^A_{\tau/2} \circ \widetilde\Phi^G_\tau \circ \Phi^A_{\tau/2}.
$$
This scheme is numerically performing very well in the absence of small nonzero singular values, see \cite{MOPP17}. The extension of our convergence proof to this situation, however, is not straightforward. First of all, a second-order scheme needs more regularity of the exact solution, in particular between the (split) vector fields and the boundary conditions. This was worked out for the full-rank Strang splitting in \cite{EO15, JL00}. The same regularity assumptions and/or modifications are also required here. The numerical example, given in \cite[Figure~2]{MOPP17} clearly shows that whenever the needed regularity is missing the order is restricted to 1.25 for a formally second-order splitting. The bottleneck, however, is the fact that the projector-splitting Strang scheme is not proven to be second-order convergent in the case of small nonzero singular values, see \cite{KLW16}.

\subsection{Further convergence results}
For the purpose of simplicity and clarity, we have restricted our convergence analysis up to now to parabolic problems and a nonlinearity $G$ that does not necessarily satisfy the boundary conditions of the involved elliptic differential operator. In this case, the quantity $AG + GA\trasp$ (see \eqref{loc_err_formula_sp}) cannot be bounded independently of the spatial grid size. This is the place where the parabolic smoothing property~\eqref{smooth_pr} enters the game. A typical instance for such a situation is the following. The matrix $A$ stems from the spatial discretization of an elliptic differential operator subject to homogeneous boundary conditions and $G(t,X(t))$ is a (spatially) smooth function that does not vanish at the boundary. 

However, there are interesting situations in which our proof still holds even if~\eqref{smooth_pr} does not hold. A typical possibility is the following one. Let the differential operator be of the form $v\cdot \nabla$, where $v$ is a given velocity vector. We thus consider a transport semigroup in a Hilbert space. This is a semigroup of contractions and satisfies~\eqref{bound} with $\omega=0$. 

If this problem is now considered with periodic boundary conditions, the quantity $AG + GA\trasp$ can be uniformly bounded if $G(t,X(t))$ is smooth in space. In this case, the parabolic smoothing property in not required and low-rank Lie--Trotter splitting is first-order convergent on compact time intervals. As this proof follows from a straightforward modification of the given proof, we do not work out the details.

\section{Differential Lyapunov equations} \label{sect:DLE}
As a special case of the stiff matrix differential equation \eqref{eq}, we consider differential Lyapunov equations (DLEs), which are of crucial importance in many applications, e.g., Kalman filtering, model reduction of linear time-varying systems, optimal filtering or numerical simulation of systems governed by stochastic partial differential equations \cite{LSS16}. 

For DLEs, the term $G(t,X)$ in \eqref{eq} is solution independent. We denote the resulting time-dependent matrix as $Q(t)$. This gives us the DLE 
\begin{align*}
 \dot{X}(t) = AX(t) + X(t)A\trasp + Q(t), \quad X(t_0) = X^0, 
\end{align*}
where $A, Q(t), X(t) \in \mathbb{C}^{m \times m}$. The matrix $Q$ and the initial data $X^0$ are symmetric and positive semidefinite. Since the DLE is linear, its exact solution exists for all times and is also symmetric and positive semidefinite.

In order to find a low-rank approximation $Y(t) \in \M$ for the solution $X(t)$ of the DLE, we follow the procedure described in Section~\ref{procedure}. 
First, we split the DLE into the following two subproblems:
	\begin{alignat*}{2}
	\dot{X}_1(t) &= AX_1(t) + X_1(t)A\trasp, \qquad &  X_1(t_0) & = X_1^0, \\
	\dot{X}_2(t) &= Q(t), & X_2(t_0) & = X_2^0.
	\end{alignat*} 
By $\tilde{\Phi}_{\tau}^Q$ we denote the flow of the second subproblem approximated by means of the projector-splitting integrator. Then, the low-rank solution is computed by the low-rank Lie--Trotter splitting integrator $\mathcal{I}_{\tau} = \PhiA \circ \tilde{\Phi}_{\tau}^Q$, see \eqref{approx_lie} with $Q$ instead of $G$.

The analysis of the global error of this scheme goes along the proofs in Section \ref{proofs}, if the DLE satisfies Assumption~\ref{hp_split} in Section~\ref{conv_theorem}. 
Since DLEs are typically stemming from parabolic partial differential equations, we assume that the matrix $A$ satisfies the properties in Assumption~\ref{hp_split}\ref{A_ass}. The inhomogeneity $Q(t)$ is not solution dependent. Thus, we have $L=0$, and Assumption~\ref{hp_split}\ref{G_ass} is satisfied.
To fulfill Assumption~\ref{hp_split}\ref{M-R}, we have to assume that $Q(t)$ is in the tangent space $\mathcal{T}_{Y}\mathcal{M}$ up to a small perturbation $R(t,Y):=Q(t) - P(Y)Q(t)$, i.e., we assume $\lVert R(t,Y) \rVert \leq \varepsilon$. This assumption is strictly related to the existence of a low-rank structure for the solution of DLEs. Some theoretical results are given in \cite{Antoulas02, Penzl00, Stillfjord08}.

Thus, we can simply apply the error analysis given in Section \ref{proofs}. The bound of the global error of the full-rank Lie--Trotter splitting integrator stays the same, i.e., $ \lVert E^n_{sp} \rVert \leq C\tau(1 +\lvert \log \tau \rvert)$ with the only difference that here the constant $C$ does not depend on $L$. Furthermore, the result given in Proposition \ref{globerr} becomes
$
\lVert E^n_{lr} \rVert \leq c_2 \varepsilon.
$
We observe that, compared to the general case, the constant $c_3$ drops here. This clearly follows from the fact that the constant $b_2$ of Lemma  \ref{locerr} is zero here. Also the constant $c_1$ appearing in the error bound for the propagated difference between the full-rank and the low-rank initial values does not depend on $L$.

Finally, we remark that the low-rank Lie--Trotter splitting \eqref{approx_lie} can be tailored to preserve symmetry and positive semidefiniteness of the solution, as explained in \cite{MOPP17}. This modification does not introduce any further difficulties in the convergence analysis, since it is only based on a different representation of the solution. A symmetric variant of the SVD-like decomposition in Subsection~\ref{proj-split} is employed in the modified algorithm.

\section{Differential Riccati equations} \label{sect:DRE}
The class of matrix differential equations of the form \eqref{eq} also includes differential Riccati equations (DREs). They play a crucial role in many applications, such as optimal and robust control problems, optimal filtering, $H_{\infty}$ control of linear time varying systems, and differential games, see \cite{AKFJ03,IK99,PUS00}. Further, several integrators based on low-rank approximations have been proposed in the past years. In particular, we mention methods based on backward differentiation formulas and Rosenbrock methods \cite{BM13,BennerMena2}. 

For DREs the nonlinearity $G$ in \eqref{eq} is quadratic and of the form 
\[ G(t,X) = Q(t)-X(t)KX(t). \]
Thus, we consider here the following initial value problem  
\begin{align}\label{DRE}
\dot{X}(t) = AX(t) + X(t)A\trasp + Q(t) - X(t)KX(t), \quad X(t_0) = X^0, 
\end{align}
where $A, Q(t), K, X(t) \in \mathbb{C}^{m \times m}$. The matrices $Q$ and $K$, and the initial value $X^0$ are symmetric and positive semidefinite.  The global existence and positive semidefiniteness of the solution is guaranteed under these conditions, see \cite{DE94}.

As for the case of DLEs, the rather general framework given in Assumption~\ref{hp_split} fits to DREs. 
Condition \ref{A_ass} is fulfilled by assuming that the matrix $A$ is the discretization of a strongly elliptic differential operator. Therefore we restrict ourselves to parabolic problems. Property \ref{G_ass} is a usual requirement in the field of differential equations. On the other hand, condition \ref{M-R} requires more care. 
Let $Y^0$ be a rank-$r$ approximation of $X^0$. Then the rank-$r$ solution of \eqref{DRE} is given by
\begin{equation*}
Y(t) =\ee^{(t-t_0) A} Y^0 \ee^{(t-t_0) A\trasp} + \int_{t_0}^t \ee^{(t-s)A} P(Y(s)) \left(Q(s)-Y(s)KY(s)\right)\ee^{(t-s)A\trasp} \dd s
\end{equation*}
for $t_0\leq t\leq T$. Assumption~\ref{hp_split}\ref{M-R} requires that the nonlinearity $G$ has a particular form when computed along a low-rank solution $Y$. 
We can take the tangential part as 
\[ M(t,Y) = P(Y)Q(t) - YKY,\] 
whereas the residual is 
\begin{equation*} 
R(t,Y) = Q(t)-P(Y)Q(t). 
\end{equation*}
To verify this, note, that the term $M(t,Y)$ is the sum of two elements of the tangent space. Indeed, $P(Y)Q(t)$ is trivially an element of the tangent space. For $YKY$ we proceed as follows. 
Making use of the explicit form of the projection recalled in \eqref{projection}, we observe that
\begin{align*}
P(Y) (YKY) & = UU\trasp (USV\trasp K Y) - UU\trasp (USV\trasp K USV\trasp) VV\trasp + (Y K USV\trasp) VV\trasp \\
& = U S V\trasp K Y - USV\trasp K USV\trasp + Y K USV\trasp \\ 
& = YKY,
\end{align*}
where we have used the fact that $U$ and $V$ have orthonormal columns. 
Since $\TM$ is a vector space we conclude that $M(t,Y) \in \TM$.  
For the residual we assume $\lVert R(t,Y) \rVert \leq \varepsilon$, see \cite{Antoulas02,Stillfjord08} for some related theoretical results.

Although we carried out the proof in the matrix setting, DREs can be also studied from an abstract different point of view, see, e.g., \cite{LT00}. A convergence analysis for a splitting method in the setting of Hilbert--Schmidt operators was proposed in \cite{HS14}. Moreover, different types of splitting for DREs were proposed in \cite{S15, S17}. 

As for DLEs, the low-rank Lie--Trotter splitting can be tailored to preserve symmetry and positive semidefiniteness of the solution. For an algorithm of our proposed method in the case of DREs, see \cite{MOPP17}.

\section{Numerical results}
The aim of this section is to illustrate the numerical behaviour of the low-rank Lie--Trotter splitting \eqref{approx_lie}. In particular, we present a numerical example to illustrate the convergence result of Theorem \ref{main_th}.
 
We study a DRE arising in optimal control for linear quadratic regulator problems. 
Thus we consider the linear control system
\[\dot{x} = Ax + u, \quad x(0) = x_0, \]
where $A \in \mathbb{R}^{m \times m}$ is the system matrix, $x \in \mathbb{R}^m$ the state variable and $u \in \mathbb{R}^m$ the control. The functional $\mathcal{J}$, that has to be minimized is given by
\[ \mathcal{J}(u,x) = \frac{1}{2}\int_{0}^{T} \Big(x(t)^{\mathsf{T}} C^{\mathsf{T}} C x(t)+u(t)^{\mathsf{T}} u(t)\Big) \dd t, \]
where $C \in \mathbb{R}^{q \times m}$ and $(\cdot)^{\mathsf{T}}$ denotes the transpose. 
Further, the optimal control is given in feedback form by $ u_{\text{opt}}(t)=-X(t)x(t)$,  where $X(t)$ is the solution of the following DRE
\begin{equation*}
\dot{X}(t) = A^{\mathsf{T}} X(t) + X(t)A + C^{\mathsf{T}} C - X(t)^2,
\end{equation*}
which is in the form of \eqref{DRE} with $Q = C^{\mathsf{T}} C$ and $K=I_m$ being the identity matrix.  

In order to consider an interesting application, we mainly follow the numerical example  presented in \cite{HS14}. The matrix $A$ arises from the spatial discretization of the diffusion operator
\[ \mathcal{D} = \partial_x \left(\alpha(x) \partial_x (\cdot) \right)  - \lambda I, \]
defined on the spatial domain $\Omega = (0,1)$ subject to homogeneous Dirichlet boundary conditions. We choose $\alpha(x) = 2+\cos(2 \pi x)$ and $\lambda = 1$. 
The finite difference discretization of the operator $\mathcal{D}$ satisfies Assumption \ref{hp_split}\ref{A_ass}. Let $q$ be odd. The matrix $C \in \mathbb{R}^{q \times m}$ is defined by taking $q$ independent vectors 
$\{1,e_1, \dots, e_{(q-1)/2}, f_1, \dots, f_{(q-1)/2}\}$,
 where 
\[ e_k(x)= \sqrt{2}\cos(2 \pi k x) \quad \text{and} \quad f_k(x) = \sqrt{2} \sin(2\pi kx), \quad k = 1, \dots, (q-1)/2\] 
are evaluated at the grid points $\{x_j\}_{j=1}^{m}$, where $x_j = \frac{j}{m+1}$.
The following results are obtained by choosing the initial value $X^0=0$, final time $T = 0.1$, $m = 200$ and $q = 9$.

\begin{figure}[htb]
	\begin{minipage}[b][4.10cm][t]{.465\textwidth}
		\centering
		\includegraphics[width = \columnwidth, height= 4.10cm]{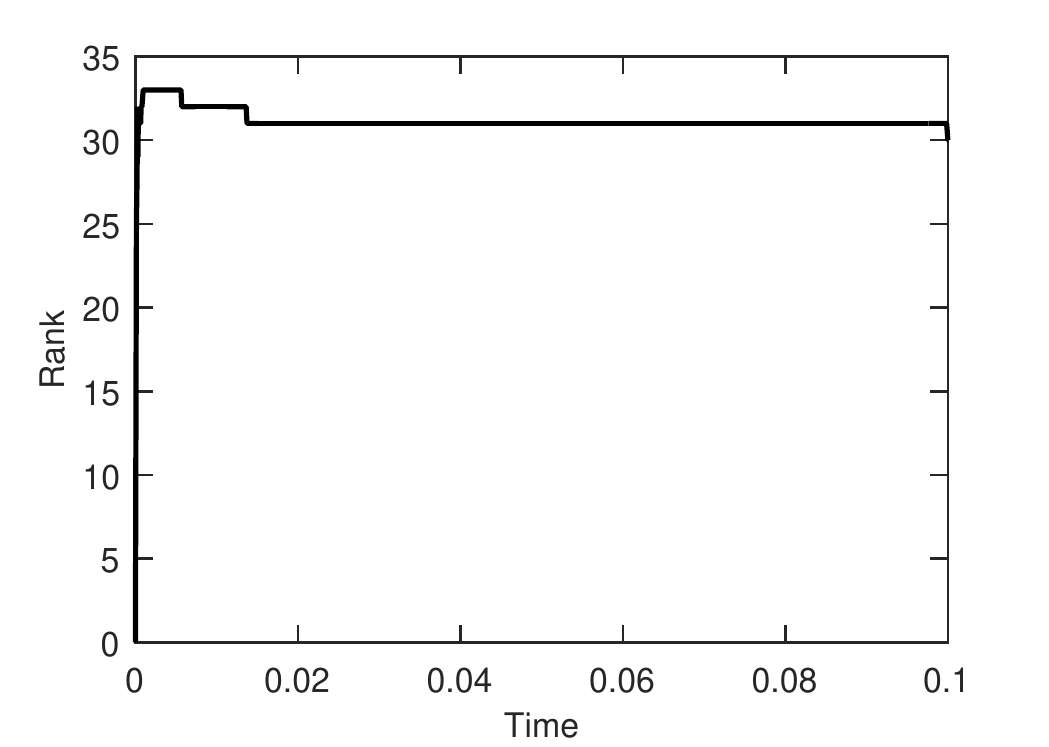}
	\end{minipage}
	\hfill
	\begin{minipage}[b][4.10cm][t]{.465\textwidth}
		\centering
		\includegraphics[width = \columnwidth, height= 3.98cm]{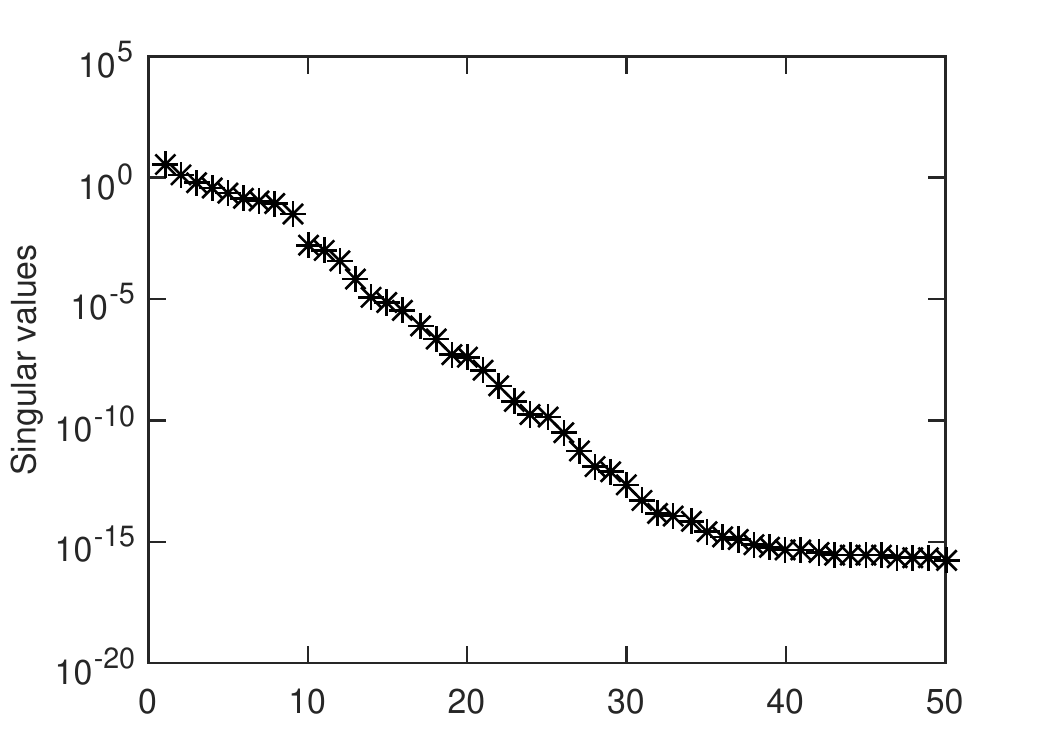}
	\end{minipage}
	\caption{Results for the considered DRE for $m= 200$. Left: Rank of the reference solution as a function of time. Right: First 50 singular values of the reference solution at $T=0.1$.}
	\label{fig:rde_1}
\end{figure}

\begin{figure}[htb]
	\centering
	\includegraphics[width = .465\columnwidth, height= 4.10cm]{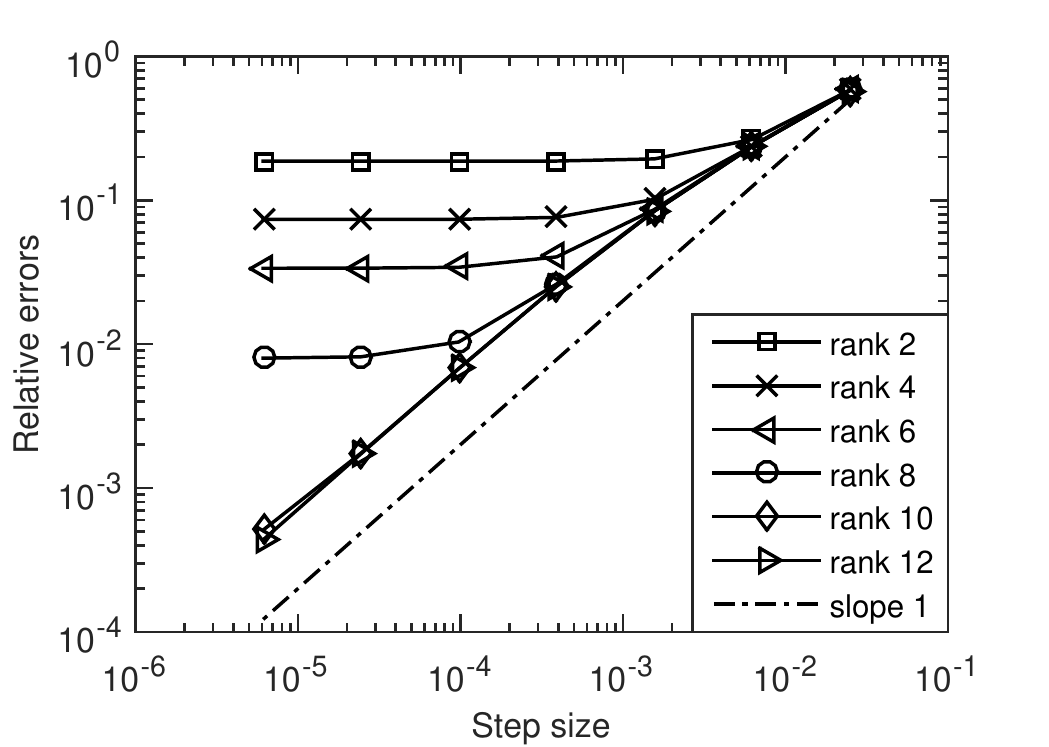}
	\caption{Errors of the low-rank Lie--Trotter splitting in the Frobenius norm as function of step size and rank at $T=0.1$ for the considered DRE for $m=200$.}
	\label{fig:rde_2}
\end{figure}

In Figure \ref{fig:rde_1} left, we show the rank of the reference solution, which is computed by DOPRI5 \cite{HNW93}. We observe that the effective rank of the solution stays low during the evolution in time. In Figure \ref{fig:rde_1} right, we plot the first 50 singular values of the solution at the final integration time. 
In Figure \ref{fig:rde_2}, the error behaviour of the low-rank Lie--Trotter splitting \eqref{approx_lie} is illustrated. We observe that the error is composed by two different contributions. The choice of a small approximation rank results in stagnation of the error. On the other hand, if the low-rank error becomes small enough, one observes the usual order of convergence one for the outer Lie--Trotter splitting. This is consistent with the convergence result given in Theorem \ref{main_th}.

\section*{Acknowledgments}
We thank the referees for their helpful comments, which improved the presentation of this paper. 

\bibliography{references}

\begin{thebibliography}{10}

\bibitem{AKFJ03}
H.~Abou-Kandil, G.~Freiling, V.~Ionescu, and G.~Jank.
\newblock {\em Matrix {R}iccati Equations in Control and Systems Theory}.
\newblock Birkh\"auser, Basel, 2003.

\bibitem{ALH11}
A.~H. Al-Mohy and N.~J. Higham.
\newblock Computing the action of the matrix exponential, with an application
  to exponential integrators.
\newblock {\em SIAM J. Sci. Comput.}, 33:488--511, 2011.

\bibitem{Antoulas02}
A.~Antoulas, D.~Sorensen, and Y.~Zhou.
\newblock On the decay rate of {H}ankel singular values and related issues.
\newblock {\em Systems Control Lett.}, 46:323--342, 2002.

\bibitem{BM13}
P.~Benner and H.~Mena.
\newblock Rosenbrock methods for solving differential {R}iccati equations.
\newblock {\em IEEE Trans. Automat. Control}, 58:2950--2957, 2013.

\bibitem{BennerMena2}
P.~Benner and H.~Mena.
\newblock Numerical solution of the infinite-dimensional {LQR} problem and the
  associated {R}iccati differential equations.
\newblock {\em J. Numer. Math.}, 26:1--20, 2018.

\bibitem{CKOR16}
M.~Caliari, P.~Kandolf, A.~Ostermann, and S.~Rainer.
\newblock The {L}eja method revisited: backward error analysis for the matrix
  exponential.
\newblock {\em SIAM J. Sci. Comput.}, 38, 2016.

\bibitem{DE94}
L.~Dieci and T.~Eirola.
\newblock Positive definiteness in the numerical solution of {R}iccati
  differential equations.
\newblock {\em Numer. Math.}, 67:303--313, 1994.

\bibitem{Gronwall86}
J.~Dixon and S.~McKee.
\newblock Weakly singular discrete {G}ronwall inequalities.
\newblock {\em Z. Angew. Math. Mech.}, 66:535--544, 1986.

\bibitem{EO13}
L.~Einkemmer and A.~Ostermann.
\newblock Exponential integrators on graphic processing units.
\newblock In {\em 2013 International Conference on High Performance Computing
  Simulation (HPCS)}, pages 490--496, 2013.

\bibitem{EO15}
L.~Einkemmer and A.~Ostermann.
\newblock Overcoming order reduction in diffusion-reaction splitting. {P}art 1:
  Dirichlet boundary conditions.
\newblock {\em SIAM J. Sci. Comput.}, 37:A1577--A1592, 2015.

\bibitem{EN06}
K.-J. Engel and R.~Nagel.
\newblock {\em A Short Course on Operator Semigroups}.
\newblock Springer, New York, 2006.

\bibitem{HLW00}
E.~Hairer, C.~Lubich, and G.~Wanner.
\newblock {\em Geometric Numerical Integration. Structure-Preserving Algorithms
  for Ordinary Differential Equations}.
\newblock Springer, Berlin, Heidelberg, second edition, 2000.

\bibitem{HNW93}
E.~Hairer, S.~P. N{\o}rsett, and G.~Wanner.
\newblock {\em Solving Ordinary Differential Equations I: Nonstiff Problems}.
\newblock Springer, New York, second edition, 1993.

\bibitem{HS14}
E.~Hansen and T.~Stillfjord.
\newblock Convergence analysis for splitting of the abstract differential
  {R}iccati equation.
\newblock {\em SIAM J. Numer. Anal.}, 52:3128--3139, 2014.

\bibitem{HM96}
U.~Helmke and J.~B. Moore.
\newblock {\em Optimization and Dynamical Systems}.
\newblock Springer, London, 1996.

\bibitem{IK99}
A.~Ichikawa and H.~Katayama.
\newblock Remarks on time-varying {$H_{\infty}$} {R}iccati equations.
\newblock {\em Sys. Cont. Lett.}, 37:335--345, 1999.

\bibitem{JL00}
T.~Jahnke and C.~Lubich.
\newblock Error bounds for exponential operator splitting.
\newblock {\em BIT}, 40:735--744, 2000.

\bibitem{KLW16}
E.~Kieri, C.~Lubich, and H.~Walach.
\newblock Discretized dynamical low-rank approximation in the presence of small
  singular values.
\newblock {\em SIAM J. Numer. Anal.}, 54:1020--1038, 2016.

\bibitem{KL07}
O.~Koch and C.~Lubich.
\newblock Dynamical low-rank approximation.
\newblock {\em SIAM J. Matrix Anal. Appl.}, 29:434--454, 2007.

\bibitem{LSS16}
N.~Lang, J.~Saak, and T.~Stykel.
\newblock Balanced truncation model reduction for linear time-varying systems.
\newblock {\em Math. Comput. Model. Dyn. Syst.}, 22:267--281, 2016.

\bibitem{LT00}
I.~Lasiecka and R.~Triggiani.
\newblock {\em Control Theory for Partial Differential Equations: Volume 1,
  Abstract Parabolic Systems: Continuous and Approximation Theories}.
\newblock Cambridge University Press, 2000.

\bibitem{lubich15tii}
C.~Lubich.
\newblock Time integration in the multiconfiguration time-dependent {H}artree
  method of molecular quantum dynamics.
\newblock {\em Appl. Math. Res. Express.}, 2015:311--328, 2015.

\bibitem{LO14}
C.~Lubich and I.~V. Oseledets.
\newblock A projector-splitting integrator for dynamical low-rank
  approximation.
\newblock {\em BIT}, 54:171--188, 2014.

\bibitem{LOV15}
C.~Lubich, I.~V. Oseledets, and B.~Vandereycken.
\newblock Time integration of tensor trains.
\newblock {\em SIAM J. Numer. Anal.}, 53:917--941, 2015.

\bibitem{LRSV13}
C.~Lubich, T.~Rohwedder, R.~Schneider, and B.~Vandereycken.
\newblock Dynamical approximation by hierarchical {T}ucker and tensor-train
  tensors.
\newblock {\em SIAM J. Matrix Anal. Appl.}, 34:470--494, 2013.

\bibitem{LVW18}
C.~Lubich, B.~Vandereycken, and H.~Walach.
\newblock Time integration of rank-constrained {T}ucker tensors.
\newblock {\em SIAM J. Numer. Anal.}, 56:1273--1290, 2018.

\bibitem{MLQ02}
R.~I. McLachlan and G.~R.~W. Quispel.
\newblock Splitting methods.
\newblock {\em Acta Numer.}, 11:341--434, 2002.

\bibitem{MOPP17}
H.~Mena, A.~Ostermann, L.-M. Pfurtscheller, and C.~Piazzola.
\newblock Numerical low-rank approximation of matrix differential equations.
\newblock {\em J. Comput. Appl. Math.}, 340:602--614, 2018.

\bibitem{Nonnemacher08}
A.~Nonnenmacher and C.~Lubich.
\newblock Dynamical low-rank approximation: applications and numerical
  experiments.
\newblock {\em Math. Comput. Simulation}, 79:1346--1357, 2008.

\bibitem{O11}
I.~V. Oseledets.
\newblock Tensor-train decomposition.
\newblock {\em SIAM J. Sci. Comput.}, 33:2295--2317, 2011.

\bibitem{OT09}
I.~V. Oseledets and E.~E. Tyrtyshnikov.
\newblock Breaking the curse of dimensionality, or how to use {SVD} in many
  dimensions.
\newblock {\em SIAM J. Sci. Comput.}, 31:3744--3759, 2009.

\bibitem{P83}
A.~Pazy.
\newblock {\em Semigroups of Linear Operators and Applications to Partial
  Differential Operators}.
\newblock Springer, New York, 1983.

\bibitem{Penzl00}
T.~Penzl.
\newblock Eigenvalue decay bounds for solutions of {L}yapunov equations: The
  symmetric case.
\newblock {\em Systems Control Lett.}, 40:139--0144, 2000.

\bibitem{PUS00}
I.~R. Petersen, V.~A. Ugrinovskii, and A.~V. Savkin.
\newblock {\em Robust Control Design Using $H^{\infty}$ Methods}.
\newblock Springer, London, 2000.

\bibitem{ReRo93}
M.~Renardy and R.~Rogers.
\newblock {\em An Introduction to Partial Differential Equations}.
\newblock Springer, New York, second edition, 2004.

\bibitem{S92}
Y.~Saad.
\newblock Analysis of some {K}rylov subspace approximations to the matrix
  exponential operator.
\newblock {\em SIAM J. Numer. Anal.}, 29:209--228, 1992.

\bibitem{S15}
T.~Stillfjord.
\newblock Low-rank second-order splitting of large-scale differential {R}iccati
  equations.
\newblock {\em IEEE Trans. Automat. Control}, 60:2791--2796, 2015.

\bibitem{S17}
T.~Stillfjord.
\newblock Adaptive high-order splitting schemes for large-scale differential
  {R}iccati equations.
\newblock {\em Numer. Algorithms}, 78:1129--1151, 2018.

\bibitem{Stillfjord08}
T.~Stillfjord.
\newblock Singular value decay of operator-valued differential {L}yapunov and
  {R}iccati equations.
\newblock {\em SIAM J. Control Optim.}, 56:3598--3618, 2018.

\bibitem{T66}
L.~R. Tucker.
\newblock Some mathematical notes on three-mode factor analysis.
\newblock {\em Psychometrika}, 31:279--311, 1966.

\end{thebibliography}

\end{document}